\documentclass[12pt]{article}
\usepackage[cp1251]{inputenc}
\usepackage[english]{babel}
\usepackage{amssymb}
\usepackage{amsmath}
\usepackage{amsmath,amsthm,amsfonts,color}

\theoremstyle{plain}
\newtheorem{theorem}{Theorem}
\theoremstyle{plain}
\newtheorem{lemma}{Lemma}
\theoremstyle{plain}
\newtheorem{cor}{Corollary}
\theoremstyle{plain}
\newtheorem{definition}{Definition}
\theoremstyle{plain}
\newtheorem{remark}{Remark}

\begin{document}

\begin{center}\large\bf
S.~A.~Chumachenko, S.~F.~Lukomskii, P.~A.~Terekhin
\\
\Large
Riesz Bounds  \\ of Spline Affine Systems
\end{center}
\noindent
 N.G.\ Chernyshevskii Saratov State University, Russia\\
  	  chumachenkosergei@gmail.com \\
  LukomskiiSF@info.sgu.ru\\
   terekhinpa@mail.ru 
\begin{center}
Abstracts
\end{center}

We construct a family of spline affine Riesz bases, i.e. sequences of dilations and translations generated by the special spline functions
$\psi_m$, and we prove that their Riesz bounds are independent of $m$.
We put $\psi_0=\chi$ as the Haar step function and every following function $\psi_{m+1}$ is obtained by integrating the previous function  $\psi_m$ with consequent antiperiodization. We give a representation of the spline $\psi_m$ as a finite sum of Rademacher chaos series
and we use a notion of a simple Walsh spectrum of a function in connection with orthogonality of affine systems.


\section*{Introduction}

It is well-known that affine Haar systems, i.e. sequences of dilations and translations of functions on the unit interval,
possess many useful properties such as localization, coherency and basicity
(the last property holds under certain conditions on the generating function).
The classical dyadic Haar system is the only orthonormal basis among such affine systems.
Nevertheless, there is a wide class of functions that generates affine Riesz bases.
Note that here we are talking about generating functions with support in the unit interval.

The Haar system consists of discontinuous functions.
The problem of the inclusion of the Haar system in a more general class of similar functional sequences consisting of smooth functions has a long history.
Franklin and Ciesielski systems were obtained in this way but they do not have localization and coherence properties.
These systems may be constructed by integration of Haar functions and using the Gram~-~Schmidt process.
Therefore, these systems are difficult to use in numerical applications.
Another way to solve this problem is to use wavelet theory methods (see, for example,~\cite{Chu},~\cite{Dau},~\cite{NPS}).
Within the framework of this theory spline wavelets of Stromberg \cite{Str}, Battle~-~Lemarier \cite{Bat}, Chui~-~Wang \cite{CW}
based on the concept of $B$-splines were constructed.

In this paper we construct Haar type affine Riesz bases in the space $L^2(0,1)$ generated by a special spline functions $\psi_m$,
$m=1,2,\dots$.
The generating functions $\psi_m$ of these systems are m-th integral of the Walsh functions $w_{2^{m+1}-1}$ (Figures 1 and 2).
We will use another representation of the generating function $\psi_m$ under which the following function $\psi_{m+1}$  is obtained by integrating the previous function $\psi_m$ with consequent antiperiodization. The zero mean value of the generating function is a necessary condition for an affine system to be a basis. So antiperiodic functions $\psi_m$ are considered (non-strictly speaking, the Haar step function can be regarded as antiperiodic). Thus for recursive integration we need an appropriate number of antiperiodizations and the Walsh function $w_{2^{m+1}-1}$ arises naturally.

First of all, we show that each of the constructed spline affine systems is a Riesz basis (Theorem~\ref{t.1}).
Secondly, we prove that the Riesz bounds of all such affine systems can be chosen universally, i.e. independently of $m$ (Theorem~\ref{t.2}).
This means that the spline affine systems have global stability for arbitrary smoothness.

\section{ Main results}\label{s.1}

Put $r(t)=\mbox{sign}\sin(2\pi t)$. Then $r_k(t)=r(2^kt)$, $k=0,1,\dots$, is a Rademacher sequence.
For $m=1,2,\dots$ we define the $1$-periodic spline $\psi_m(t)$ of order $m$ with knots $\frac{1}{2^{m+1}}\mathbb{Z}$ by
\begin{equation}\label{eq.1.1}
\frac{d^m}{dt^m}\psi_m(t)=\varkappa_m\prod_{k=0}^mr_k(t), \qquad t\notin\frac{1}{2^{m+1}}\mathbb{Z},
\end{equation}
\begin{equation}\label{eq.1.2}
\frac{d^{\mu}}{dt^{\mu}}\psi_m(t)\biggr|_{t=0}=0, \qquad \mu=0,\dots,m-1,
\end{equation}
where $\varkappa_m=2^{m(m+5)/2}$.
Figures 1 and 2 show the graphs of the functions $\psi_1(t){\chi_{[0,1]}}$ and $\psi_m(t){\chi_{[0,1]}}$, $m>1$.

\unitlength=0.60mm
 \begin{picture}(80,90)
 \put(-5,40){\line(1,0){85}}
 \put(-5,-2){\line(0,1){89}}
 \put(-5,40){\line(1,2){20}}
 \put(15,80){\line(1,-2){40}}
 \put(55,0){\line(1,2){20}}

 \put(31,33){\small $ \frac12$}
\put(73,34){\small $ 1$}
 \multiput(15,40)(0,6){7}{\line(0,1){4}}
 \multiput(55,40)(0,-6){7}{\line(0,-1){4}}
 \multiput(-5,80)(5,0){4}{\line(1,0){4}}
 \put(14,33){\small $ \frac14$}
 \put(52,44){\small $ \frac34$}
 \put(-9,78){\small $ 2$}
  \put(30,-7){{\bf Fig. 1.}\ $\psi_1(t)$}

 \end{picture}\qquad\qquad
  \unitlength=0.60mm
 \begin{picture}(84,90)
 \put(0,40){\line(1,0){84}}
 \put(0,-2){\line(0,1){88}}
\qbezier(0,40)(5,40)(10,60)
\qbezier(10,60)(15,80)(20,80)
\qbezier(20,80)(25,80)(30,60)
\qbezier(30,60)(35,40)(40,40)

\qbezier(39,40)(45,40)(50,20)
\qbezier(50,20)(55,0)(60,0)
\qbezier(60,0)(65,0)(70,20)
\qbezier(70,20)(75,40)(80,40)
\put(36,33){\small $ \frac12$}
\put(78,34){\small $ 1$}
 \multiput(20,40)(0,6){7}{\line(0,1){4}}
 \multiput(60,40)(0,-6){7}{\line(0,-1){4}}
 \multiput(0,80)(5,0){4}{\line(1,0){4}}
 \put(19,33){\small $ \frac14$}
 \put(57,44){\small $ \frac34$}
 \put(-5,78){\small $ 2$}
 \put(30,-8){{\bf Fig. 2.} \ $\psi_m(t)$}

\end{picture}\\

\vspace{3mm}
If we put formally $m=0$ in equation (\ref{eq.1.1}) then we obtain Haar system
$(\chi_n)_{n\ge0}$ that is a sequence of dilations and translations of the function $\chi=r\chi_{[0,1]}$.

\begin{definition}\label{d.0}
Let $f:\mathbb{R}\to\mathbb{R}$ and $\mbox{supp}\,f\subset[0,1]$.
The system of dilations and translations of the function $f$ is a sequence of functions
$$
f_0(t)=1, \qquad f_n(t)=2^{k/2}f(2^kt-j), \quad n=1,2,\dots,
$$
where $n=2^k+j$, $k=0,1,\dots$ and $j=0,\dots, 2^k-1$.
\end{definition}

\begin{definition}\label{d.1}
The  sequence of dilations and translations  $(\psi_{m,n})_{n\ge0}$ of the function $\psi_m\chi_{[0,1]}$ is called a spline affine system of order $m$.
\end{definition}

\begin{theorem}\label{t.1}
For any $m=1,2,\dots$ the spline affine system $(\psi_{m,n})_{n\ge0}$ is a Riesz basis in $L^2(0,1)$.
\end{theorem}

We recall that the sequences $(\psi_n)_{n\ge0}$ of elements of Hilbert space $H$ is called a Riesz basis if there exists an orthonormal basis $(e_n)_{n\ge0}$ of $H$ and a bounded invertible linear operator (isomorphism) $T:H\to H$ such that
$\psi_n=Te_n$ for all  $n=0,1,\dots$.
Positive constants $A\le B$ are called  Riesz bounds of a sequence $(\psi_n)_{n\ge0}$ if for any $c=(c_n)_{n\ge0}\in\ell^2$ the inequalities
$$
A\|c\|_2\le\biggl\|\sum_{n=0}^{\infty}c_n\psi_n\biggr\|\le B\|c\|_2.
$$
are satisfied. It is clear that $A=\|T^{-1}\|^{-1}$ and $B=\|T\|$ for Riesz bases.
The main question of this article: are there universal (independent of $m=1,2,\dots$) Riesz bounds for considered spline affine systems?
A positive response is given by the following main theorem.

\begin{theorem}\label{t.2}
If $c=(c_n)_{n\ge0}\in \ell^2$ then
$$
\frac{1}{10}\|c\|_2\le\biggl\|\sum_{n=0}^{\infty}c_n\psi_{m,n}\biggr\|\le\frac{19}{10}\|c\|_2
$$
for any $m\in\mathbb N$.
\end{theorem}

To prove these results we represent the spline $\psi_m$ as a finite sum of Rademacher chaos series of odd orders $d=2s+1$, $s=0,\dots,m$.
Note that for $m=1$ Theorem 1 was proved in \cite{T2}.

\section{Preliminaries}\label{s.2}

At first, we give the inductive construction of splines $\psi_m$.
Let
$$
Vf(t)=\int_0^tf(s)\,ds
$$
be the Volterra operator.
For  $1$-periodic function  $f(t)$ we define dilation~-~modulation operators $W_0,W_1$ in the following way
$$
W_0f(t)=f(2t), \qquad W_1f(t)=r(t)f(2t).
$$
Further, we put
$$
U=4W_1V.
$$
Note that the function  $f=W_1g$ is $\frac12$-antiperiodic, i.e. $f(t+\frac12)=-f(t)$. If a function $f$ is antiperiodic and $(f,r)=1$ then
$(Uf,r)=1$.

\begin{lemma}\label{l.1}
For any $m\in\mathbb{N}$ we have $\psi_m=U^mr$.
\end{lemma}

\begin{proof}
If a 1-periodic integrable function $f$ has zero mean value $\int_0^1f(t)\,dt=0$ then
\begin{equation}\label{eq.2.1}
VW_0f(t)=\frac12 W_0Vf(t), \qquad VW_1f(t)=\frac12 W_1Vf(t).
\end{equation}
Indeed, since $\int_0^1f(t)\,dt=0$ it follows that the functions $VW_if$ and $W_iVf$, $i=0,1$ are absolutely continuous and equal to zero at the points $\frac12\mathbb{Z}$.
Since the derivative  $(W_iVf)'(t)=2r^i(t)f(2t)$ is equal to $2(VW_if)'(t)$ a.e. in $\mathbb{R}$, it follows that the equalities~\eqref{eq.2.1} hold for all $t\in\mathbb{R}$.

Using \eqref{eq.1.1} -- \eqref{eq.2.1} we get
$$
U^mr=\underbrace{4W_1V\dots4W_1V}_mr=4^m2^{1+\ldots+m}V^mW_1^mr=
\varkappa_mV^mW_1^mr=
$$
$$
=\varkappa_mV^m\prod_{k=0}^mr_k=V^m\frac{d^m}{dt^m}\psi_m=\psi_m.
$$
\end{proof}

Lemma~\ref{l.1} shows that $\psi_{m+1}=U\psi_m $, i.e. each successive spline function is an antiperiodization of integral of the previous spline. Obviously, the 0-th function $r=W_11$ is antiperiodic too (hereinafter $1$ is the function identically equal to one).

\begin{definition}\label{d.2}
Rademacher chaos of order  $d\in\mathbb{N}$ is a family of functions $r_{i_1}\dots r_{i_d}$, $0\le i_1<\ldots<i_d$, where $r_{i_j}$
are Rademacher functions. The set of all functions
$$
f=\sum_{0\le i_1<\ldots<i_d}a_{i_1,\dots,i_d}r_{i_1}\dots r_{i_d}, \qquad \sum_{0\le i_1<\ldots<i_d}|a_{i_1,\dots,i_d}|^2<\infty,
$$
we will denote  $\mathcal{R}_{ch}^d$.
\end{definition}

It is not difficult to see that $r_{i_1}\dots r_{i_d}=W_0^{k_1}W_1\dots W_0^{k_d}W_11$, \ $k_1,\dots,k_d\ge0$, where $i_1=k_1$, $i_2=k_1+k_2+1$, \ldots, $i_d=k_1+\dots+k_d+d-1$. It follows that any function  $f\in\mathcal{R}_{ch}^d$ may be written in the form
$$
f=\sum_{k_1,\dots,k_d\ge0}c_{k_1,\dots,k_d}W_0^{k_1}W_1\dots W_0^{k_d}W_11, \qquad \sum_{k_1,\dots,k_d\ge0}|c_{k_1,\dots,k_d}|^2<\infty.
$$

\begin{lemma}\label{l.2}
For any $m\in\mathbb R$ the following representation holds
\begin{equation}\label{eq.2.2}
\psi_m=r-\sum_{s=1}^mf_{m,s}, \qquad m=1,2,\dots,
\end{equation}
where $f_{m,s}\in\mathcal{R}_{ch}^d$ for $d=2s+1$, $s=1,\dots,m$.
\end{lemma}

\begin{proof}
The proof is by induction on $m$. Let $m=1$. We define an auxiliary function
$$
\lambda=\sum_{k=0}^{\infty}\frac{1}{2^{k+1}}W_0^kW_11=\sum_{k=0}^{\infty}\frac{r_k}{2^{k+1}}.
$$
It is well known that $\lambda(t)=1-2t$, $0<t<1$ is a unique (accurate to factor)  continuous function on $ (0,1) $ which can be represented by a Rademacher series (chaos of order $ d = 1 $). Let us show that
\begin{equation}\label{eq.2.3}
Ur=r-W_1^2\lambda.
\end{equation}
Indeed, on the one hand we have
$$
4Vr(t)=\begin{cases}
4t, & 0\le t\le \frac12,\\
4(1-t), & \frac12\le t \le 1.
\end{cases}
$$
On the other hand
$$
W_1\lambda(t)=\begin{cases}
1-4t, & 0\le t\le \frac12,\\
4t-3, & \frac12\le t \le 1.
\end{cases}
$$
Therefore $4Vr=1-W_1\lambda$. This implies that  $4W_1Vr=W_11-W_1^2\lambda$, and \eqref{eq.2.3} is proved.
It is clear that $f_{1,1}=W_1^2\lambda\in\mathcal{R}_{ch}^3$.

Suppose that equation \eqref{eq.2.3} is true for some $m$.
Using the equation $\psi_ {m + 1} = U \psi_m $ we obtain from \eqref{eq.2.2}
$$
\psi_{m+1}=Ur-\sum_{s=1}^mUf_{m,s}.
$$
We know that $4Vr=1-W_1\lambda$,  $W_0^{k_d}1=1$, and $f_{m,s}\in\mathcal{R}_{ch}^d$ for $d=2s+1$. Using it and  \eqref{eq.2.1} we get
\begin{multline*}
U f_{m,s}=\frac{1}{2^{d-1}}\sum_{k_1,\dots,k_d\ge0}\frac{c_{k_1,\dots,k_d}}{2^{k_1+\dots+k_d}}W_1W_0^{k_1}\dots W_1W_0^{k_d}4Vr\\
=\frac{1}{2^{d-1}}\sum_{k_1,\dots,k_{d-1}\ge0}\frac{1}{2^{k_1+\dots+k_{d-1}}}\biggl(\sum_{k_d\ge0}\frac{c_{k_1,\dots,k_d}}{2^{k_d}}\biggr)
W_1W_0^{k_1}\dots W_1W_0^{k_{d-1}}r\\
-\frac{1}{2^{d-1}}\sum_{k_1,\dots,k_d\ge0}\frac{c_{k_1,\dots,k_d}}{2^{k_1+\dots+k_d}}W_1W_0^{k_1}\dots W_1W_0^{k_d}W_1\lambda
=g_{m,s}-h_{m,s+1}.
\end{multline*}
It is clear that $g_{m,s}\in\mathcal{R}_{ch}^d$ and  $h_{m,s+1}\in\mathcal{R}_{ch}^{d+2}$.
Therefore
$$
\psi_{m+1}=r-W_1^2\lambda-\sum_{s=1}^m(g_{m,s}-h_{m,s+1})=r-\sum_{s=1}^{m+1}f_{m+1,s},
$$
where we denoted
\begin{gather}\label{eq.2.4}
f_{m+1,1}=g_{m,1}+W_1^2\lambda, \\ \label{eq.2.5}
f_{m+1,s}=g_{m,s}-h_{m,s}, \qquad s=2,\dots,m, \\ \label{eq.2.6}
f_{m+1,m+1}=-h_{m,m+1}.
\end{gather}
To complete the proof, we note that $f_{m+1,s}\in\mathcal{R}_{ch}^d$ for $d=2s+1$, $s=1,\dots,m+1$.
\end{proof}

To prove the main theorem it is enough to study the dynamics of chaos of order $d=3$ in the representation of the spline $\psi_m$.

\begin{lemma}\label{l.3}
For any $m\in \mathbb N$
$$
f_{m,1}=\gamma_mW_1^2r+W_1^2\lambda,
$$
where $\gamma_m=\frac29(1-\frac{1}{4^{m-1}})$.
\end{lemma}

\begin{proof}
In the proof of Lemma 2 we have obtained the equality $f_{1,1}=W_1^2\lambda$. It is Lemma \ref{l.3} for $m=1$.
Suppose that the equality $f_{m,1}=\gamma_mW_1^2r+W_1^2\lambda$ is true for some $m$. Write it in the form
$$
f_{m,1}=W_1^2(\gamma_m+\tfrac12)r+W_1^2\sum_{k=1}^{\infty}\frac{1}{2^{k+1}}W_0^kr.
$$
Using notation from Lemma \ref{l.2} we have
$c_{0,0,0}=\gamma_m+\frac12$,
$c_{0,0,k}=\frac{1}{2^{k+1}}$ for $k\ge1$, and $c_{k_1,k_2,k_3}=0$
for $k_1\neq 0$ or $k_2\neq 0$.
It implies that
$$
g_{m,1}=\frac14\sum_{k_1,k_2\ge0}\frac{1}{2^{k_1+k_2}}\biggl(\sum_{k_3\ge0}\frac{c_{k_1,k_2,k_3}}{2^{k_3}}\biggr)W_1W_0^{k_1}W_1W_0^{k_2}r.
$$
Since $\sum_{k=1}^{\infty}\frac{1}{2^{2k+1}}=\frac16$, it follows that
$$
\sum_{k_3\ge0}\frac{c_{k_1,k_2,k_3}}{2^{k_3}}=\begin{cases}
\gamma_m+\frac12+\frac16,   & k_1=k_2=0,\\
0,   & { k_1+k_2}\neq 0.
\end{cases}
$$
So we have the equality
$$
g_{m,1}=\frac{\gamma_m+\frac23}{4}W_1^2r.
$$
Combining this equality and \eqref{eq.2.4} we get
$$
f_{m+1,1}=\frac{\gamma_m+\frac23}{4}W_1^2r+W_1^2\lambda.
$$
Since $\gamma=\frac29$ is the unique fixed point of the mapping $\gamma\mapsto\frac{\gamma+\frac23}{4}$, it follows that the sequence
$\gamma_m=\frac29(1-\frac{1}{4^{m-1}})$ is a solution of the recurrent relations
$$
\gamma_1=0, \qquad \gamma_{m+1}=\frac{\gamma_m+\frac23}{4}, \qquad m=1,2,\dots.
$$
This concludes the proof.
\end{proof}

\section{Simple Walsh spectra and orthogonality of affine systems}\label{s.3}

Let us consider all possible products of operators $W_0,W_1$:
$$
W^{\alpha}=W_{\alpha_0}\dots W_{\alpha_{k-1}}, \qquad \alpha=(\alpha_0,\dots,\alpha_{k-1})\in\mathbb{A}:=\bigcup_{k=0}^{\infty}\{0,1\}^k.
$$
It is clear that
$$
(W^{\alpha}r)(t)=r^{\alpha_0}(2^0t)\dots r^{\alpha_{k-1}}(2^{k-1}t)r(2^kt)=
r_k(t)\prod_{\nu=0}^{k-1}r_{\nu}^{\alpha_{\nu}}(t).
$$
Therefore the family $(W^{\alpha}r)_{\alpha\in\mathbb{A}}$ is the Walsh system $(w_n)_{n\ge1}$ (without the function $w_0=1$) in Paley enumeration $n=\sum_{\nu=0}^{k-1}\alpha_{\nu}2^{\nu}+2^k$.
For any $1$-periodic function $f$ the family $(W^{\alpha}f)_{\alpha\in\mathbb{A}}$ is called an affine Walsh system, generated by $f$ (\cite{AT},~\cite{T1}).

By analogy we consider the dilation~-~translation operators
$$
S_0=\frac{W_0+W_1}{\sqrt{2}}, \qquad S_1=\frac{W_0-W_1}{\sqrt{2}}
$$
and their various products
$$
S^{\alpha}=S_{\alpha_0}\dots S_{\alpha_{k-1}}, \qquad \alpha=(\alpha_0,\dots,\alpha_{k-1})\in\mathbb{A}.
$$
It it easy to see that the family $(S^{\alpha}f)_{\alpha\in\mathbb{A}}$ coincides with the sequence $(f_n)_{n\ge1}$ of dilation and translation of a function $f\chi_{[0,1]}$. In particular, $(S^{\alpha}r)_{\alpha\in\mathbb{A}}$ is the Haar system $(\chi_n)_{n\ge1}$ (without the function $\chi_0=1$) in natural enumeration $n=\sum_{\nu=1}^k\alpha_{k-\nu}2^{\nu-1}+2^k$.
For any $1$-periodic function $f$ the family $(S^{\alpha}f)_{\alpha\in\mathbb{A}}$ is also called the affine Haar system, generated by $f$.

Let $L_0^2$ denote the set of all $1$-periodic functions $f\in L^2(0,1)$ for which $\int_0^1f(t)\,dt=0$.

\begin{definition}\label{d.4}
The Walsh spectrum of $f\in L_0^2$ is the set of multi-indices
$$
\mbox{Spec}_W(f)=\{\beta\in\mathbb{A}:(f,W^{\beta}r)\neq0\}.
$$
\end{definition}
Let $\alpha\beta=(\alpha_0,\dots,\alpha_{k-1},\beta_0,\dots,\beta_{l-1})$ be a concatenation of multi-indices $\alpha,\beta\in\mathbb{A}$
and $|\alpha|=k$ be a length of multi-index $\alpha$.

\begin{definition}\label{d.5}
We say that the function $f\in L_0^2$ has a simple Walsh spectrum if from the equality $\alpha\beta=\alpha'\beta'$, it follows that $\alpha=\alpha',\ \beta=\beta'$, where $\alpha,\alpha'\in\mathbb{A}$ and $\beta,\beta'\in\mbox{Spec}_W(f)$.
\end{definition}

\begin{lemma}\label{l.4}
If the function $f\in L_0^2$, $f\neq0$ has a simple Walsh spectrum, then affine Walsh and Haar systems $(W^{\alpha}f)_{\alpha\in\mathbb{A}}$, $(S^{\alpha}f)_{\alpha\in\mathbb{A}}$ are orthogonal.
\end{lemma}

\begin{proof}
Firstly, we note that if Walsh spectra of $g$ and $h$ are disjoint then these functions are orthogonal: $(g,h)=0$.
Then we consider the Fourier~-~Walsh series of the function $f\in L_0^2$
$$
f=\sum_{\beta\in\mbox{Spec}_W(f)}(f,W^{\beta}r)W^{\beta}r.
$$
By the definition of operators $W^{\alpha}$ we have
$$
W^{\alpha}f=\sum_{\beta\in\mbox{Spec}_W(f)}(f,W^{\beta}r)W^{\alpha\beta}r.
$$
It follows that
$$
\mbox{Spec}_W(W^{\alpha}f)=\{\alpha\beta\}_{\beta\in\mbox{Spec}_W(f)}.
$$
Show that $\mbox{Spec}_W(W^{\alpha}f)\cap\mbox{Spec}_W(W^{\alpha'}f)=\emptyset$
for $\alpha\ne\alpha'$. Indeed, suppose that $\mbox{Spec}_W(W^{\alpha}f)\cap\mbox{Spec}_W(W^{\alpha'}f)\neq\emptyset$. The common point of these spectra is $\alpha\beta=\alpha'\beta'$, where $\alpha,\alpha'\in\mathbb{A}$ and $\beta,\beta'\in\mbox{Spec}_W(f)$. Using the definition of a simple spectrum we get $\alpha=\alpha'$, which is impossible.
Thus we have shown that $(W^{\alpha}f,W^{\alpha'}f)=0$ for $\alpha\ne\alpha'$. Since $\|W^{\alpha}f\|=\|f\|>0$, it follows that the affine Walsh system $(W^{\alpha}f)_{\alpha\in\mathbb{A}}$ is orthogonal.

Show that the affine Haar system $(S^{\alpha}f)_{\alpha\in\mathbb{A}}$ is orthogonal too.
Since
$$
w_{2^k+i}=\sum_{j=0}^{2^k-1}\varepsilon_{i,j}\chi_{2^k+j}, \qquad i=0,\dots,2^k-1,
$$
where $(\varepsilon_{i,j})$ is a unitary Walsh matrix, it follows that
$$
\chi_{2^k+j}=\sum_{i=0}^{2^k-1}\varepsilon_{i,j}w_{2^k+i}, \qquad j=0,\dots,2^k-1.
$$
In our notation, the last equation can be written as
$$
S^{\alpha}r=\sum_{|\beta|=k}\varepsilon_{\alpha,\beta}W^{\beta}r, \qquad |\alpha|=k.
$$
Therefore
$$
S^{\alpha}f=\sum_{|\beta|=k}\varepsilon_{\alpha\beta}W^{\beta}f, \qquad |\alpha|=k.
$$
Let  $(W^{\alpha}f)_{\alpha\in\mathbb{A}}$ be an orthogonal system. Then
$$
(S^{\alpha}f,S^{\alpha'}f)=
\sum_{|\beta|=k}\sum_{|\beta'|=k'}
\varepsilon_{\alpha\beta}\varepsilon_{\alpha'\beta'}(W^{\beta}f,W^{\beta'}f)=0
$$
for  $k'\neq k$. If $k'=k$ then $(W^{\beta}f,W^{\beta'}f)\ne 0 $ for  $\beta'=\beta$ only. Therefore
$$
(S^{\alpha}f,S^{\alpha'}f)=
\|f\|_2^2\sum_{|\beta|=k}\varepsilon_{\alpha\beta}\varepsilon_{\alpha'\beta}=\delta_{\alpha,\alpha'}\|f\|_2^2.
$$
This completes the proof.
\end{proof}

\section{Proofs of theorems}\label{s.4}

To prove Theorems~\ref{t.1} and~\ref{t.2} we will use an uniform method.
For any $f\in L_0^2$ we define a transform $T_f$ of the Haar and Walsh systems into the affine Haar and Walsh systems respectively
by the equivalence relations
$$
T_fS^{\alpha}r=S^{\alpha}f, \qquad T_fW^{\alpha}r=W^{\alpha}f, \qquad \alpha\in\mathbb{A}.
$$
Extend this transform $T_f$ by linearity on a linear manifold of all Haar or Walsh polynomials (both are the same) with zero mean.

By definition the transform $T_f$ commutes with operators $S_0,S_1$ and  with
$W_0,W_1$ equivalently. Note that $T_f$ is not bounded in $L_0^2$  in general case.
The conditions under which the operator $T_f$ is bounded were obtained in~\cite{T2}.
It is clear that for orthogonal affine systems $(S^{\alpha}f)_{\alpha\in\mathbb{A}}$, $(W^{\alpha}f)_{\alpha\in\mathbb{A}}$ we have
\begin{equation}\label{eq.4.1}
\|T_fx\|=\|f\|\|x\|, \qquad x\in L_0^2,
\end{equation}
and consequently $\|T_f\|=\|f\|$.

We show that if $f$ is antiperiodization of a function from $\mathcal{R}_{ch}^d$, i.e.
$$
f=\sum_{k_1,\dots,k_d\ge 0}c_{k_1,\dots,k_d}W_1W_0^{k_1}\dots W_1W_0^{k_d}r
\in \mathcal{R}_{ch}^{d+1},
$$
then equality \eqref{eq.4.1} is carried out.
Indeed, for the Walsh spectrum of $f$ we have
$$
\mbox{Spec}_W(f)\subset\{(1,{\bf0}_{k_1},\dots,1,{\bf0}_{k_d})\}_{k_1,\dots,k_d\ge0},
$$
where  ${\bf 0}_k=(0,\dots,0)$, $k\ge0$ is multi-index consisting of $k$ zeros. If $k=0$ then we set that ${\bf 0}_k$ is an empty index.
It follows from Definition~\ref{d.5} that the function $f$ has a simple Walsh spectrum.
By Lemma~\ref{l.4} affine Walsh and Haar affine systems generated by $f$ are orthogonal.

Write equation~\eqref{eq.2.2} in the operator form
$$
T_{\psi_m}=I-\sum_{s=1}^mT_{f_{m,s}}.
$$
Since $\psi_m$ and all functions $f_{m,s}\in\mathcal{R}_{ch}^{2s+1}$ are antiperiodic, it follows that $\|T_{f_{m,s}}\|=\|f_{m,s}\|$. Therefore
$$
\|I-T_{\psi_m}\|\le\sum_{s=1}^m\|f_{m,s}\|, \qquad m=1,2,\dots.
$$
Let us prove that
\begin{equation}\label{eq.4.2}
\sum_{s=1}^m\|f_{m,s}\|<\frac{9}{10}
\end{equation}
for all $m=1,2,\dots$. In Lemma~\ref{l.3} we obtained the equality $f_{m,1}=\gamma_mW_1^2r+W_1^2\lambda$ from which we find
$$
\|f_{m,1}\|=\biggl((\gamma_m+\tfrac12)^2+\sum_{k=1}^{\infty}\frac{1}{4^{k+1}}\biggr)^{1/2}
<\sqrt{\biggl(\frac29+\frac12\biggr)^2+\frac{1}{12}}=\frac79.
$$
Using the notation of Lemma~\ref{l.2}, we have
\begin{multline*}
\|g_{m,s}\|=\frac{1}{2^{d-1}}\biggl(\sum_{k_1,\dots,k_{d-1}\ge0}\frac{1}{4^{k_1+\dots+k_{d-1}}}
\biggl(\sum_{k_d\ge0}\frac{c_{k_1,\dots,k_d}}{2^{k_d}}\biggr)^2\biggr)^{1/2}\\
\le\frac{1}{2^{d-1}}\biggl(\frac43\sum_{k_1,\dots,k_d\ge0}c_{k_1,\dots,k_d}^2\biggr)^{1/2}=\frac{2\|f_{m,s}\|}{4^s\sqrt{3}}
\le\frac{2\|f_{m,s}\|}{16\sqrt{3}},
\end{multline*}
for $s=2,...,m$. In the last inequality we have used $d=2s+1$. Now we consider the equality
$$
h_{m,s+1}=\frac{1}{2^{d-1}}\sum_{k_1,\dots,k_d\ge0}\frac{c_{k_1,\dots,k_d}}{2^{k_1+\dots+k_d}}W_1W_0^{k_1}\dots W_1W_0^{k_d}W_1\lambda
$$
as expansion into series with respect to an affine Walsh system  with the generating function $W_1\lambda$.
Since $\lambda\in\mathcal{R}_{ch}^1$, it follows that this affine system is orthogonal and $\|{W_1\lambda}\|=\|\lambda\|=\frac{1}{\sqrt{3}}$. Therefore
$$
\|h_{m,s+1}\|=\frac{1}{2^{d-1}\sqrt{3}}\biggl(\sum_{k_1,\dots,k_d\ge0}\frac{c_{k_1,\dots,k_d}^2}{4^{k_1+\dots+k_d}}\biggr)^{1/2}
\le\frac{\|f_{m,s}\|}{4^s\sqrt{3}}\le\frac{\|f_{m,s}\|}{16\sqrt{3}}.
$$
Now we need to estimate the norm $\|h_{m,2}\|=\|h_{m,s+1}\|$ for $s=1$. To make this we will use Lemma \ref{l.3}. Under the proof of this lemma we obtained $c_{0,0,0}=\gamma_m+\frac12$, $c_{0,0,k}=\frac{1}{2^{k+1}}$ for  $k\ge1$,
and  $c_{k_1,k_2,k_3}=0$ for $|k_1|+|k_2|>0$. Using these values of
$c_{k_1,k_2,k_3}$ we find
\begin{multline*}
h_{m,2}=\frac14\sum_{k_1,k_2,k_3\ge0}\frac{c_{k_1,k_2,k_3}}{2^{k_1+k_2+k_3}}W_1W_0^{k_1}W_1W_0^{k_2}W_1W_0^{k_3}W_1\lambda\\
=\frac14\biggl(c_{0,0,0}W_1^4\lambda+\sum_{k=1}^{\infty}\frac{c_{0,0,k}}{2^k}W_1^3W_0^kW_1\lambda\biggr)\\
=\frac14W_1^3\biggl((\gamma_m+\tfrac12)W_1\lambda+\sum_{k=1}^{\infty}\frac{1}{2^{2k+1}}W_0^kW_1\lambda\biggr).
\end{multline*}
Therefore
$$
\|h_{m,2}\|=\frac{1}{4\sqrt{3}}\biggl((\gamma_m+\tfrac12)^2+\sum_{k=1}^{\infty}\frac{1}{4^{2k+1}}\biggr)^{1/2}
<\frac{1}{4\sqrt{3}}\sqrt{\biggl(\frac29+\frac12\biggr)^2+\frac{1}{60}}=a,
$$
where  $a=\frac{\sqrt{218}}{36\sqrt{15}}<0,106$.
Using  \eqref{eq.2.5}, \eqref{eq.2.6}, and obtained estimates we find
$$
\sum_{s=2}^{m+1}\|f_{m+1,s}\|\le\|h_{m,2}\|+\sum_{s=2}^m(\|g_{m,s}\|+\|h_{m,s+1}\|)
<a+\frac{\sqrt{3}}{16}\sum_{s=2}^m\|f_{m,s}\|.
$$
It follows that for all $m$
$$
\sum_{s=2}^m\|f_{m,s}\|<a\sum_{n=0}^{\infty}\biggl(\frac{\sqrt{3}}{16}\biggr)^n
<\frac{0,106}{1-0,11}<0,12.
$$
Finally, we have
$$
\sum_{s=1}^m\|f_{m,s}\|<\frac79+0,12<\frac{9}{10}
$$
and inequality~\eqref{eq.4.2} is proved. It means that $\|I-T_{\psi_m}\|<\frac{9}{10}$. It follows that operator $T_{\psi_m}$ is isomorphism of $L_0^2$, and a spline affine system  $(\psi_{m,n})_{n\ge0}$ is a Riesz basis. We have also
$$
\biggl\|\sum_{n=0}^{\infty}c_n(\chi_n-\psi_{m,n})\biggr\|\le\frac{9}{10}\|c\|_2.
$$
It follows that  $A=\frac{1}{10}$, $B=\frac{19}{10}$ are universal Riesz bounds for spline affine systems. This completes the proof.
$\square$

\begin{remark}\rm
Granados~\cite{Gra} considered the functions
$$
\rho_m=\widetilde{U}^m1, \qquad m=1,2,\dots,
$$
where $\widetilde{U}=4VW_1$. She proved that affine Walsh system (the author called it Walsh wavelets)
\begin{equation}\label{eq.5.1}
\rho_m(2^{k+1}t)w_n(t),  \qquad k=0,1,\dots, \qquad n=2^k,\dots,2^{k+1}-1,
\end{equation}
is a frame when  $m=1$ and  $m=2$.
\end{remark}

The following statement can be obtained from Theorems~\ref{t.1} and~\ref{t.2}.

\begin{cor}
For any  $m\ge3$ function system~\eqref{eq.5.1} along with the function $1$ is a Riesz basis.
The Riesz bounds of all systems~\eqref{eq.5.1}  can be selected independently of $m$.
\end{cor}

\begin{proof}
It suffices to note that the system~\eqref{eq.5.1} coincides with affine Walsh system $(W^{\alpha}\psi_m)_{\alpha\in\mathbb{A}}$,
generated by spline $\psi_m$.
\end{proof}

This research was carried out with the financial
support of  the Russian Foundation for Basic Research (grant
no.~16-01-00152)

\end{document}